\documentclass[10pt]{article}
\usepackage{amsmath,amsxtra,latexsym,amsthm,amssymb,amscd,amsfonts,verbatim,cprotect}
\usepackage{orcidlink}
\usepackage{graphicx} % Required for inserting images
\usepackage[mathscr]{eucal}
\usepackage{amsfonts}
\usepackage{graphics,graphicx,array}
\usepackage{subcaption}
\usepackage{color}
\usepackage{listings}
\usepackage{hyperref}
\usepackage{algorithm}
\usepackage{algorithmic}
\usepackage{listings, color}
\usepackage{tikz}
\definecolor{codegray}{rgb}{0.5,0.5,0.5}
\definecolor{codeblue}{rgb}{0.1,0.1,0.7}
\definecolor{codegreen}{rgb}{0,0.6,0}

\makeatletter
\let\@fnsymbol\@arabic
\makeatother

\newtheorem{ex}{\bf Example}[section]
\newtheorem{remark}{\bf Remark}[section]

\newtheorem{theorem}{\bf Theorem}[section]

\title{Similarity-based fuzzy clustering scientific articles: \\potentials and challenges from mathematical and computational perspectives}

\author{Vu Thi Huong\,\orcidlink{0009-0007-4869-0505}\footnote{
		Digital Data and Information for Society, Science, and Culture,
		Zuse Institute Berlin, 14195 Berlin, Germany; and
		Institute of Mathematics, Vietnam Academy of Science and Technology, 10072 Hanoi, Vietnam. Email: huong.vu@zib.de}, \, Ida Litzel\,\orcidlink{ 0009-0000-0166-7809}\footnote{Digital Data and Information for Society, Science, and Culture,
		Zuse Institute Berlin, Germany. Email: Litzel@zib.de}, \, and  Thorsten Koch\,\orcidlink{0000-0002-1967-0077}\footnote{Software and Algorithms for Discrete Optimization, Technische Universit\"at Berlin, Germany; and Applied Algorithmic Intelligence Methods, Zuse Institute Berlin, Germany. Email: koch@zib.de}}
\begin{document}
\maketitle
\begin{center}
{Dedicated to Professor Yurii Nesterov on the occasion of his 70th birthday}
\end{center}
\begin{quote}
\noindent {\bf Abstract.} Fuzzy clustering, which allows an article to belong to multiple clusters with soft membership degrees, plays a vital role in analyzing publication data. This problem can be formulated as a constrained optimization model, where the goal is to minimize the discrepancy between the similarity observed from data and the similarity derived from a predicted distribution. While this approach benefits from leveraging state-of-the-art optimization algorithms, tailoring them to work with real, massive databases like OpenAlex or Web of Science -- containing about 70 million articles and a billion citations -- poses significant challenges. We analyze potentials and challenges of the approach from both mathematical and computational perspectives. Among other things, second-order optimality conditions are established, providing new theoretical insights, and practical solution methods are proposed by exploiting the problem’s structure. Specifically, we accelerate the gradient projection method using GPU-based parallel computing to efficiently handle large-scale data.



\end{quote}

\section{Introduction}\label{Sect_Introduction} 
The exponential growth of scientific literature has intensified the need for effective computerized methods to map and understand the structure of research fields. Bibliometric analysis, which quantitatively examines scholarly communication, plays a crucial role in this endeavor; see, e.g., \cite{Fortunato_etal_18}. A key task within bibliometrics is clustering scientific publications into coherent groups that reflect topical, methodological, or intellectual similarities. This facilitates applications such as literature recommendation \cite{Beel_etal_16} or identification of emerging research areas \cite{Klavans_Boyack_16}.

Popular clustering methods, including the Louvain \cite{Blondel_etal_08} and Leiden \cite{Traag_etal_19} algorithms, typically assign each publication to a single cluster. However, in reality, areas often overlap, as seen, for example, in the intersection of theoretical computer science and mathematics, and interdisciplinary research can easily span multiple topics. This motivates the use of \textit{fuzzy clustering}, which allows for overlapping cluster memberships. To the best of our knowledge, we found no prior studies that have specifically applied fuzzy clustering methods to the clustering of scientific publications -- a gap that this study aims to address.

Scientific documents exhibit rich structures through both citation-based linkages and textual content, offering multiple complementary perspectives for assessing similarity and forming clusters; see \cite{Xie_Waltman_17, Waltman_etal_20}. Citation-based similarities, such as direct citations, co-citations, and bibliographic coupling, capture intellectual relationships, while textual features reflect semantic closeness. These diverse features provide a foundation for clustering but also raise the question of which fuzzy clustering techniques are best suited to integrate such heterogeneous data.

A commonly used approach in fuzzy clustering is the Fuzzy C-Means (FCM) algorithm \cite{Bezdek_81}, which has demonstrated success in various domains, including medical image segmentation \cite{Pham_etal_00} and pattern recognition \cite{Pal_Bezdek_95}. However, FCM requires the vectorization of documents -- transforming textual or citation data into numerical vectors -- prior to clustering. This preprocessing step can be computationally intensive when dealing with large-scale datasets and also tightly couples representation and clustering.%, limiting flexibility in integrating diverse data types, such as citations and text.

To overcome these limitations, we adopt the similarity-based fuzzy clustering method proposed by Nepusz et al. \cite{Nepusz_etal_08}, which models clustering as a constrained optimization problem minimizing the discrepancy between observed and predicted similarities. Unlike FCM, this approach can operate directly on pairwise similarity matrices derived from citation data, eliminating the need for vectorization. It also allows for the integration of additional features, such as semantic similarities from text, into the similarity matrix without altering the core clustering mechanism.

Nevertheless, adapting the method to large publication databases, such as OpenAlex\footnote{\href{https://openalex.org}{https://openalex.org}} and Web of Science\footnote{\href{https://www.webofscience.com}{https://www.webofscience.com}}, which comprise millions of articles and billions of citations, poses substantial challenges. Efficient and scalable algorithms are essential for handling data of this magnitude. In this paper, we address similarity-based fuzzy clustering from both mathematical and computational perspectives. We introduce parallel implementations of the gradient projection algorithm (parallel GPA) and the fast iterative shrinkage-thresholding algorithm (parallel FISTA) for large-scale processing. Both methods yield significant reductions in the clustering objective, with parallel FISTA exhibiting strong heuristic acceleration on graphs with several million nodes. On the theoretical side, we prove the convergence of parallel GPA to critical points and, for the first time, establish second-order optimality conditions, offering insights regarding the quality of the solutions.

A central contribution of our work is a scalable parallelization strategy that leverages the mathematical structure of the fuzzy clustering problem. By decoupling computations across columns and sharing compact intermediate results, our algorithms avoid costly operations on large matrices. Our CUDA-based implementation achieves high throughput, supporting clustering on graphs with millions of nodes.

%Beyond algorithmic advances, this work lays the foundation for future enhancements, including improved initialization, adaptive step size tuning, and incorporation of hard clustering heuristics to boost robustness and interpretability. Planned extensions also include integrating structured sparsity and metadata, as well as leveraging text-based similarity derived from titles and abstracts using large language models, to better capture the nuanced structure of complex real-world scientific article graphs.

%Our study underscores the promising synergy between rigorous mathematical techniques and high-performance computational power in addressing meaningful bibliometric problems. This interdisciplinary approach unlocks new avenues for scalable, interpretable clustering of scientific literature, and we believe it represents an important direction for future research.

The remainder of the paper is organized as follows. Section 2 introduces the fuzzy clustering model and the gradient projection algorithm, highlighting the computational challenges. Section 3 details our parallelization techniques, emphasizing efficient gradient computations that enable scalability to millions of articles. Section 4 presents theoretical refinements via second-order optimality conditions. Section 5 reports on experiments across synthetic and real-world datasets, demonstrating the scalability and effectiveness of parallel GPA and parallel FISTA. We summarize our contribution and outline directions for future work in Section~6.

\section{Fuzzy clustering with gradient projection}
 Given a set of $N$ scientific articles, we are interested in the problem of partitioning the set into $C$ clusters such that articles within the same cluster are more similar to each other than to articles in different clusters. To reflect the fact that an article can span multiple research domains, we allow each article to belong to multiple clusters with varying degrees of membership.

Such a partition can be described by a matrix $X = (x_{ki}) \in [0,1]^{C \times N}$
%    \[
%\begin{array}{c|cccccc}
% X & \text{article 1} & \text{article 2}  & \dots & \text{article i}  & \dots %& \text{article N} \\
%  \hline
%  \text{cluster 1} & x_{11} & x_{12}  & \dots & x_{1i} & \dots & x_{1N}\\
%  \text{cluster 2} & x_{21} & x_{22}  & \dots & x_{2i} & \dots & x_{2N} \\
%  \vdots & \vdots  & \vdots & \dots & \vdots & \dots & \vdots\\
%  \text{cluster } k & x_{k1} & x_{k2}  & \dots & {x_{ki}} & \dots & x_{kN}\\
%  \vdots & \vdots & \vdots  & \dots & \vdots & \dots & \vdots\\
%  \text{cluster C} & x_{C1} & x_{C2}  & \dots & x_{Ci} & \dots & x_{CN}\\
%\end{array}
%\]
with each $x_{ki}$ representing the membership degree of the article $i$ w.r.t. the cluster $k$. %When $x_{ki} = 0$, the article $i$ does not belong to the cluster $k$. 
When $x_{ki} = 1$, then article $i$ sole belongs to cluster $k$; otherwise 
%$x_{ki} \in (0, 1)$, 
it participates in every cluster $k$ where $x_{ki} > 0$. %in several clusters.
%, not only cluster $k$. 
For each article, we assume that the total membership degrees over $C$ clusters is $1$. Thus, a valid partition is a matrix from the set
\begin{equation}\label{constraint}
\Omega :=  \left\{X = (x_{ki}) \in [0,1]^{C \times N} \ \big | \sum_{k = 1}^C  x_{ki}  = 1, \ \forall i \in \{1, \dots, N\}\right\}.
\end{equation}

A partition can reveal a similarity measure among articles. Consider two articles $i, j \in \{1, \dots, N\}$  from a partition matrix $X \in \mathbb R^{C \times N}$. For each  $k \in \{1, \dots, C\}$, $x_{ki}$ and $x_{kj}$ are membership degrees of the articles $i$ and $j$, respectively. How much property of the cluster $k$ the two articles share can be quantified by $x_{ki}x_{kj}$, their similarity w.r.t. cluster $k$. Summing up over all $C$ clusters, we get a similarity between $i$ and $j$. This means that each partition $X \in \mathbb R^{C \times N}$ defines a similarity matrix $Y = (y_{ij}) \in  [0,1]^{N \times N}=X^TX$ % with $y_{ij}=x_{\cdot i}^Tx_{\cdot j}$ 
%components
%\begin{align}\label{predicted_similarity}
%    y_{ij}  & := \sum_{k = 1}^C x_{ki}x_{kj}
%\end{align}
with element $y_{ij}$ representing the similarity between two articles $i, j$. 
%Observe from the sum-up-to-$1$ constraint on $X$ in~\eqref{constraint} that $y_{ij} \in [0, 1]$ for all $i, j \in \{1, \dots, N\}$ . 

As our aim is to group articles based on their similarity, we suppose that an observed similarity among articles, $S = (s_{ij}) \in \mathbb R^{N \times N}$
%with $s_{ij}\in [0,1]$ for all $i, j \in \{1, \dots, N\}$, 
is given by the data, for example, through citation- or text-based measures. 
A good partition should group similar articles in the same cluster. To measure how good a partition $X$ is, we quantify how well the similarity $Y$ derived from partition $X$ approximates the observed similarity $S$ by
%\begin{align*}
%    f(X): = \displaystyle\sum_{i = 1}^N\sum_{j = 1}^N (s_{ij} - y_{ij})^2.
%\end{align*}
%Clearly, the smaller $f(X)$, the better $X$ approximates $S$. 
%Using~\eqref{predicted_similarity}, we can rewrite the above function explicitly in $X$ as
\begin{align*}%\label{objective}
    f(X) = \displaystyle\sum_{i = 1}^N\sum_{j = 1}^N (s_{ij} - y_{ij})^2=\sum_{i = 1}^N\sum_{j = 1}^N \left(s_{ij} - \sum_{k = 1}^C x_{ki}x_{kj}\right)^2 = \|S - X^TX\|_F^2,
\end{align*} 
where $\| \cdot \|_F$ stands for the matrix Frobenius norm. The problem of fuzzy clustering the set of $N$ scientific articles into $C$ clusters based on an observed similarity  $S = (s_{ij}) \in \mathbb R^{N \times N}$ now can be formulated, as in~\cite{Nepusz_etal_08}, as follows
\begin{align}\label{P}
    \min_{X \in \Omega} f(X). \tag{$\mathcal P$}
\end{align}

\noindent{\bf The gradient projection algorithm (GPA).} To design a solution method for~\eqref{P}, we first observe that the constraint set $\Omega$ is nonempty and compact while the objective function $f$ is continuous and bounded from below (by $0$) on $\mathbb R^{C \times N}$. Thus, the problem~\eqref{P} has a \textit{global optimal solution}, i.e., the (global) optimal solution set $$\mbox{Sol}\eqref{P} := \{\bar X \in \Omega \ | \ f(\bar X) \leq f(X), \   \forall X \in \Omega\}$$ is nonempty. While it is desirable to find a global optimum solution, this can be very costly, as general nonconvex optimization is $NP$-hard. Thus, it is more practical to find a \textit{local optimal solution} of \eqref{P}, i.e., a candidate $\bar X \in \Omega$ such that $f(\bar X) \leq f(X)$ for any $X \in \Omega$ and ``close enough" to $\bar X$. Since the constraint set is convex and the objective function is differentiable on the whole space, a local solution $\bar X$ has to satisfy (see, e.g., \cite[Theorem~3.24]{Rusz_06}) the  \textit{first-order necessary optimality condition}
\begin{equation}\label{opt_con}
\bar X \in \Omega \quad \mbox{and} \quad \langle \nabla f(\bar X), X - \bar X \rangle \geq 0, \quad \forall X \in \Omega,
\end{equation}
which shows a special interaction between the constraint set $\Omega$ and the gradient $\nabla f(\bar X)$ of the objective function  at $\bar X$. 

\medskip
One calls a matrix $\bar X$ satisfying~\eqref{opt_con} a \textit{critical point} of~\eqref{P}. Such a matrix can be found by the \textit{gradient projection algorithm} (GPA); see, e.g.,~\cite[Theorem~6.1]{Rusz_06}. The algorithm starts with initializing a membership matrix $X^0 \in \Omega$, then iterates over 
\begin{equation*}
	X^{n+1} := P_{\Omega}(X^n - \tau_n \nabla f(X^n)), \quad \forall n \geq 1
\end{equation*}
until a stopping condition is met. Each step of the algorithm consists of two ingredients: moving the current state $X^n$ toward the direction of the negative gradient $\nabla f(X^n)$ with a \textit{step size} $\tau_n >0$, and then projecting onto the constraint set $P_{\Omega}$ to ensure the requirements in~\eqref{constraint} are not broken. 
%To efficiently implement GPA, we will analyze further these two ingredients.

\medskip
\noindent{\bf Projection onto $\Omega$.}
 The set $\Omega$ is nonempty, closed, and convex in $\mathbb R^{C \times N}$.
 %-- the Hilbert space of matrices of size $C\times N$ with the Frobenius norm $\| \cdot\|_F$. 
 Thus, for any matrix $X \in \mathbb R^{C \times N}$, the \textit{projection} $P_{\Omega}(X)$ of $X$ onto the set $\Omega$ exists uniquely as
 \begin{equation*}
P_{\Omega}(X):= \mbox{argmin }\{\|Y - X\|_F^2 \ | \ Y \in \Omega\},
 \end{equation*}
i.e., as the minimizer of a quadratic function over the set $\Omega$. Solving this constrained optimization problem can be avoided by exploiting the structure of $\Omega$.

\medskip 
 Observe that the constraint $X \in \Omega$ means that
%\begin{itemize}
	%\item[-] 
    each element $x_{ki}$ of the matrix $X$ must lie in the interval $[0, 1]$, and
	%\item[-] 
    for each column $\mathbf{x}_i = (x_{1i}, x_{2i}, \dots, x_{Ci})^\top$, its elements must sum to 1.
	%\[
	%\sum_{k=1}^C x_{ki} = 1.
	%\]
%\end{itemize}
Therefore, $\Omega$ can be rewritten as 
\begin{equation*}
	\Omega = \underbrace{\Delta^C \times \Delta^C \times \dots \times \Delta^C}_{N \ \mbox{times}},  
\end{equation*}
where 
\[
\Delta^C := \left\{ \mathbf{y} = (y_1, y_2, \dots, y_C)^\top \in \mathbb{R}^C \ \middle| \ y_k \geq 0, \ \forall k, \ \text{and} \ \sum_{k=1}^C y_k = 1 \right\}
\] is the \textit{unit simplex} in $\mathbb R^C$. Consequently, for any $X \in \mathbb R^{C \times N}$,
 the projection of $X$ onto $\Omega$ is performed independently column-wise
\begin{equation*}
	P_{\Omega}(X) =  P_{\Delta^C}(\mathbf x_1) \times P_{\Delta^C}(\mathbf x_2) \times \dots \times P_{\Delta^C}(\mathbf x_N);
\end{equation*}
see, e.g., \cite[Proposition~29.3]{Bauschke_Bauschke_10}.

\medskip

 The unit simplex $\Delta^C$ is a nonempty, closed, and convex subset of the Euclidean space $\mathbb{R}^C$. For a given vector \(\mathbf{x} \in \mathbb R^C\), the projection $P_{\Delta^C}(\mathbf x)$ of $\mathbf x$ onto the unit simplex is the unique solution of the problem
\[
\min \| \mathbf{y} - \mathbf{x} \|_2^2 \quad \text{subject to} \quad \mathbf{y} \in \Delta^C.
\]
Note that if $\mathbf x \in \Delta^C$, then  $P_{\Delta^C}(\mathbf x) = \mathbf x$ which means that the projection of $\mathbf x$ onto the unit simplex is itself. In general, by using Lagrange multipliers to deal with the constraints of the above minimization problem, one can find the projection by, for example, the Duchi algorithm~(Algorithm \ref{Duchi_alg}; see~\cite{Duchi_etal_08}).

\medskip
To find the projection of a given vector $\mathbf x$ onto the unit simplex using the Duchi algorithm, we first sort the components of the vector \( \mathbf{x} \) in descending order. This allows the algorithm to operate on the largest components first, which simplifies the problem. Then, we find a threshold value \( \tau \) that ensures the resulting vector sums up to $1$ while maintaining nonnegative components. Finally, after determining the threshold, we subtract it from the components of \(\mathbf{x}\) and set any values that become negative to zero.

\begin{algorithm}
\caption{(Duchi algorithm -- projection onto the unit simplex -- see~\cite{Duchi_etal_08})}
\begin{algorithmic}[1] \label{Duchi_alg}
\STATE \textbf{Input:} A vector \( \mathbf{x} = [x_1, x_2, \dots, x_C] \)
\STATE Sort the components of \( \mathbf{x} \) in descending order: $x_{\text{sorted}} = [x_{(1)}, x_{(2)}, \dots, x_{(C)}]$;\\
\STATE Compute the cumulative sum of the sorted components: \( S_k = \sum_{i=1}^{k} x_{(i)} \) for \( k = 1, 2, \dots, C \);
\STATE Find the largest \( k \) such that:
$x_{(k)} - \frac{S_k - 1}{k} \geq 0$;
\STATE Set the threshold \( \tau = \frac{S_k - 1}{k} \);
\STATE For each component \( i \), compute the projected value: $
y_i = \max(x_i - \tau, 0)
$.
\STATE \textbf{Output:} The projected vector \( \mathbf{y} = [y_1, y_2, \dots, y_C] \)
\end{algorithmic}
\end{algorithm}

\medskip
\noindent{\bf Gradient update.}

\medskip
%Note that the formula~\eqref{gradient_element} can be written in a compact form as
For each $X \in [0,1]^{C\times N}$, the gradient is a matrix of size $C \times N$ given by

\begin{equation}\label{gradient_matrix}
\nabla f(X) = -4 X (S - X^TX), \quad \forall X \in \mathbb R^{C \times N}.
\end{equation} 
Updating the gradient using~\eqref{gradient_matrix} involves:
    \begin{itemize}
        \item[-] matrix multiplication \(X^T X\) (size \(N \times N\)) with  \(\mathcal{O}(CN^2)\) operations;
        \item[-] subtraction \(S - X^T X\) (size \(N \times N\)) with with  \(\mathcal{O}(N^2)\) operations;
        \item[-] matrix multiplication \(X(S - X^T X)\) (size \(C \times N\)) with  \(\mathcal{O}(CN^2)\) operations;
        \item[-] scalar multiplication by \(-4\).
    \end{itemize}
Therefore, each gradient update costs \(\mathcal{O}(CN^2)\). Combining this with the cost $\mathcal{O}(N C \log C)$ for computing the projection ($N$ columns, each column involves $C \log C$ operations for the sorting algorithm) we can estimate the \textit{per-iteration cost for the gradient projection method} as
\begin{equation*}
\mathcal{O}(CN^2 + N C \log C). 
\end{equation*}
This is extremely expensive for the target data with $N$ from $50$ to $250$ million articles, leading us to the question of \textit{how to implement the gradient projection method efficiently}.

\section{Parallelization}
The fact that the projection onto $\Omega$ is done column-wise suggests that we access columns of the gradient matrix $\nabla f(X)$ at each step. Updating $\nabla f(X)$ using~\eqref{gradient_matrix} and then extracting its $N$ columns to perform $N$ projections onto the unit simplex is prohibitive for large $N$. 
%We will now analyze the potential to update the gradient in parallel.

\medskip
Let $X \in \mathbb R^{C \times N}$ and let $\mathbf{x}_i = (x_{1i}, x_{2i}, \dots, x_{Ci})^\top$ stand for its $i$-th column. Then %formula~\eqref{gradient_element} for the 
gradient of $f$ at $X$ can be written element-wise as
\begin{equation*}
\dfrac{\partial f}{\partial x_{ki}} =  \sum_{j = 1}^N -4 (s_{ij} - \mathbf x_i^\top\mathbf x_j)x_{kj}, \quad \forall k = 1, 2, \dots, C, \ i = 1, 2, \dots, N.
\end{equation*}
This means that the element at row $k$ and column $i$ of $\nabla f(X)$ is
\begin{equation*}
    -4 \times [\mbox{row } k \mbox{ of } X] \times [\mbox{column } i  \mbox{ of } S - \mbox{ column } i \mbox{ of } X^\top X].
\end{equation*}
Thus, column $i$ of $\nabla f(X)$ is
\begin{align*}
    -4 X \times [\mbox{column } i  \mbox{ of } S - \mbox{ column } i \mbox{ of } X^\top X].
\end{align*}
Since column $i$ of $X^\top X$ is $X^\top \mathbf x_i$, we derive that 
\begin{align}\label{gradient_column}
 \mbox{column } i \mbox{ of } \nabla f(X) = -4 (X \mathbf s_i - X X^\top \mathbf x_i),
\end{align}
where $\mathbf s_i$ stands for the column $i$ of the matrix $S$. 

\begin{algorithm}
\caption{(Parallel GPA for fuzzy clustering)}
\begin{algorithmic}[1] \label{GPA_para}
\STATE \textbf{Input:} similarity matrix $S \in [0,1]^{N \times N}$,  cluster-num $C >0$, step size $\tau_n>0$\\
\STATE \textbf{Initialize:} membership matrix $X^0 = [\mathbf x_1^0, \mathbf x_2^0, \dots, \mathbf x_N^0]\in \Omega$;\\ \qquad \qquad \quad ($\mathbf x_i^0$ is the $i$-th column of $X^0$)
\STATE for each iteration $n := 0, \dots,$ max-iter \& a stopping criterion is not met:
\STATE  \quad $\mbox{share}^{n} := X^n (X^n)^\top$;
\STATE \quad for $i := 1, \dots, N$ (in parallel) do:
\STATE \quad \quad \quad $\mbox{\bf grad}_i^{n} := -4 (X^n \mathbf s_i - \mbox{share}^{n} \mathbf x_i^n)$; \quad \qquad \qquad \qquad \, [gradient update]\\
\STATE \quad \quad  \quad $\mathbf x_i^{n+1} := P_{\Delta^C} (\mathbf x_i^n - \mbox{step size} \times \mbox{\bf grad}_i^{n})$; \quad \qquad \qquad  [Duchi algorithm]
%\STATE \quad end for;
%\STATE end for;
\STATE \textbf{Output:} membership matrix $X^{n+1} = [\mathbf x_1^{n+1}, \mathbf x_2^{n+1}, \dots, \mathbf x_N^{n+1}]$
\end{algorithmic}
\end{algorithm}

\medskip
 Note that $S$ is typically a sparse matrix where each column $\mathbf s_i$ of $S$ contains only a few nonzeros, so the computation for $X \mathbf s_i$ can be done fast. The cost for $X X^\top$ is $\mathcal{O}(N C^2)$ operations, which is linear w.r.t. to the data size $N$, as $C$ is much smaller than $N$. The \textit{overall cost for updating each column of the gradient is $\mathcal{O}(N)$}, which is more manageable than $\mathcal{O}(N^2)$ of updating the whole gradient at once. More importantly, formula~\eqref{gradient_column} tells us that the gradient update can be implemented column-wise in parallel, matching with the parallelization of the projection. Note also that the computation for $X X^\top$ needs to be performed once and is shared among the column updates. 
 Therefore, we propose Algorithm~\ref{GPA_para}, a parallel gradient projection algorithm for solving~\eqref{P}.

 \medskip
The following theorem guarantees the non-increasing property of the objective function and the convergence of the iterative matrices generated by Algorithm~\ref{GPA_para} when the step size is chosen to be sufficiently small.

\begin{theorem}\label{gpa_conv_thm}
 There exists $\bar \tau > 0$ such that, for any $\tau \in (0, \bar \tau)$, the sequence $\{X^n\}$ generated by Algorithm~\ref{GPA_para} from any starting point $X^0 \in \Omega$ with constant step size $\{\tau_n\}\equiv \tau$ has the following properties:
\begin{itemize}
    \item[{\rm (i)}] The sequence $\{f(X^n)\}$ is non-increasing. In addition, it holds that $$f(X^{n+1}) < f(X^n), \quad \forall n \geq 0,$$ unless $X^n$ is a critical point of ($\mathcal P$).
    \item[{\rm (ii)}] The sequence $\{\|X^{n+1} - X^n\|_F\}$ converges to $0$; and thus, any accumulation point of $\{X^n\}$ is a critical point of ($\mathcal P$).
\end{itemize}
\end{theorem}
\begin{proof}
 Let $\bar\tau =  1/ L$ with $L:= 4\|S\|_2 + 12N$ and \( \| S \|_2 \) the spectral norm of \( S \) (the largest singular value of \( S \)). We claim that the the operator $X \mapsto \nabla f(X)$ is $L$-Lipschitz over the bounded set $\Omega$, i.e.,
\begin{equation}\label{lip_grad}
    \|\nabla f(X) - \nabla f(Y) \|_F \leq L \|X - Y\|_F, \quad \forall X, Y \in \Omega.
\end{equation}
The assertions of the theorem then follow from~\cite[Theorem~9.14]{Beck_14} and~\cite[Theorem~6.1]{Rusz_06}. 

To justify~\eqref{lip_grad}, let $X$ and $Y$ be arbitrary in $\Omega$. Using the formula~\eqref{gradient_matrix}, the submultiplicative property of the Frobenius norm, and the relation between the  Frobenius norm and the spectral norm, we have
\begin{align}\label{lip_grad1}
\|\nabla f(X) - \nabla f(Y) \|_F & = 4\|X(S - X^\top X) - Y(S - Y^\top Y)\|_F \nonumber\\
& \leq 4 \|(X - Y)S \|_F + 4 \|XX^\top X - YY^\top Y\|_F \nonumber\\
& \leq 4 \|S\|_2 \|X - Y \|_F + 4 \|XX^\top X - YY^\top Y\|_F.
\end{align}
We estimate the second term on the right-hand side of the last inequality as
\begin{align}\label{lip_grad2}
\|XX^\top X - YY^\top Y\|_F & =     \|XX^\top X - XX^\top Y +  XX^\top Y - YY^\top Y\|_F \nonumber\\
& \leq \|XX^\top X - XX^\top Y\|_F + \|XX^\top Y - YY^\top Y\|_F \nonumber\\
& \leq \|XX^\top (X -  Y)\|_F + \|(XX^\top - YY^\top)Y\|_F.
\end{align}
Thus, as before, we get
\begin{align}\label{lip_grad3}
\|XX^\top (X -  Y)\|_F & \leq \|XX^\top\|_F  \|X -  Y\|_F \nonumber\\
& \leq \|X\|^2_F  \|X -  Y\|_F
\end{align} 
and  
\begin{align}\label{lip_grad4}
\|(XX^\top - YY^\top)Y\|_F & \leq \|XX^\top - YY^\top\|_F\|Y\|_F \nonumber\\
& \leq \left(\|XX^\top - XY^\top\|_F + \|XY^\top - YY^\top\|_F\right)\|Y\|_F \nonumber\\
& \leq \left( \|X\|_F\|(X-Y)^\top\|_F + \|X - Y\|_F\|Y^\top\|_F \right)\|Y\|_F \nonumber\\
&= \left(\|X\|_F \|Y\|_F + \|Y\|_F^2\right) \|X - Y\|_F.
\end{align}

It follows from~\eqref{lip_grad1}--\eqref{lip_grad4} that
\begin{align*}
\|\nabla f(X) - \nabla f(Y) \|_F \leq 4\left ( \|S\|_2 + \|X\|^2_F + \|X\|_F \|Y\|_F + \|Y\|_F^2 \right) \|X - Y\|_F.
\end{align*}
Thus, by noticing that $\|X\|_F \leq \sqrt{N}$, $\|Y\|_F \leq \sqrt{N}$ as $X, Y \in \Omega$ we get~\eqref{lip_grad}.
%and complete the proof of the theorem.  
\end{proof}

\medskip
\noindent {\bf Loss Update.}
The first assertion of Theorem~\ref{gpa_conv_thm} suggests that a sufficient decrease of the loss $f(X)$ can be used as a stopping criterion for Algorithm~\ref{GPA_para}. This requires tracking the loss value at each iteration. Directly computing the loss via the formula $f(X) = \|S - X^\top X\|_F^2$ is expensive, since, as before, it needs to access the dense matrix $X^TX \in \mathbb R^{N \times N}$. To overcome this, we will exploit the shared matrix $XX^\top \in \mathbb R^{C \times C}$ and the vectors $X\mathbf{s_1}, X\mathbf{s_2}, \dots, X\mathbf{s_N}$ already computed when updating the gradient in parallel.

\medskip
 Let $\langle \cdot, \cdot \rangle_F$ stand for the \textit{inner product} corresponding the the Frobenius norm, i.e., $\langle X, Y \rangle_F: = \mbox{trace}(X^\top Y)$ for any two same-size real matrices $X, Y$ and $\|X\|_F = \sqrt{\langle X, X \rangle_F}$ for any real matrix $X$. By the property of the inner product, we have $f(X) = \|S - X^TX\|_F^2 = \|S\|_F^2 + \|X^TX\|_F^2 - 2 \langle X^\top X, S \rangle_F$.
%\begin{align}\label{loss}
%f(X) & = \|S - X^TX\|_F^2  \nonumber\\
%& = \|S\|_F^2 + \|X^TX\|_F^2 - 2 \langle X^\top X, S \rangle_F.
%\end{align}
Thus, to update $f(X)$, we need to compute two values  $\|X^\top X\|_F^2$ and $\langle X^\top X, S \rangle_F$. 
%Regarding the first value, observe that
Given that $\|X^\top X\|_F^2=\|XX^\top\|_F^2$
%\begin{align}\label{loss1}
%   \|X^\top X\|_F^2  & =   \mbox{trace}\left(\big( X^\top X\big)^\top \big(X^\top X\big) \right) \nonumber\\& = \mbox{trace} \left( X^\top X  X^\top X \right)\nonumber\\ & =  \mbox{trace} \left( X X^\top X  X^\top \right) \nonumber\\ & = \|XX^\top\|_F^2,
% \end{align}
 %where the first and last equalities are by definition of the norm, the second one follows from the symmetric property of $X^\top X$, and the third one is from the commutative property of the trace.
%The formula~\eqref{loss1} means
we can compute the first value 
%$\|X^\top X\|_F^2$ 
using $XX^T$, which is of size $C \times C$ and hence its Frobenius norm can be obtained directly by definition.
\begin{algorithm}
\caption{(Parallel GPA for fuzzy clustering with loss update)}
\begin{algorithmic}[1] \label{GPA_para_loss}
\STATE \textbf{Input:} similarity matrix $S \in \mathbb R^{N \times N}$,  cluster-num $C >0$, step size $\tau>0$;
\STATE \textbf{Initialize:} membership matrix $X^0 = [\mathbf x_1^0, \mathbf x_2^0, \dots, \mathbf x_N^0]\in \Omega$;\\
\qquad \qquad  \quad \ $\mbox{loss}^{-1} := N^2$ (largest value possible for the loss);
\STATE for each iteration $n := 0, \dots,$ max-iter:
\STATE  \quad $\mbox{share}^{n} := X^n (X^n)^\top$; %, $\|\mbox{share}^{n}\|_F^2$;
\STATE \quad for $i := 1, \dots, N$ (in parallel) do:
\STATE \quad \quad $\mbox{prod}_i^{n}  := \quad \ \langle X^n \mathbf s_i, \mathbf x_i^n \rangle$; \qquad \qquad \qquad \qquad \qquad \; [vector inner product]\\
%\STATE \quad end for;
\STATE  \quad $\mbox{merge}^{n} := \mbox{prod}_1^{n} + \mbox{prod}_2^{n} + \dots + \mbox{prod}_N^{n}$; \qquad \qquad \qquad \quad  [scalar addition]\\
\STATE  \quad  $\mbox{loss}^{n} \quad := \|S\|_F^2 + \|\mbox{share}^{n}\|_F^2 - 2 \times \mbox{merge}^{n}$; \qquad \qquad \quad \; [loss calculation]
\STATE  \quad if $\mbox{loss}^{n-1} - \mbox{loss}^{n} \leq \mbox{tol}$:
\STATE \qquad \quad for $i := 1, \dots, N$ (in parallel) do:
\STATE \quad \quad \qquad $\mathbf x_i^{n+1} :=\mathbf x_i^n$; 
%\STATE \qquad \quad end for;
\STATE \quad \quad  \quad  break;
\STATE  \quad  else:
\STATE \qquad \quad for $i := 1, \dots, N$ (in parallel) do:
\STATE \quad \quad \qquad $\mbox{\bf grad}_i^{n}:= -4 (X^n \mathbf s_i - \mbox{share}^{n} \mathbf x_i^n)$;  \quad \qquad \qquad \quad \; \ [gradient update]\\
\STATE \quad \quad  \qquad $\mathbf x_i^{n+1}:= P_{\Delta^C} (\mathbf x_i^n - \mbox{step size} \times \mbox{\bf grad}_i^{n})$; \quad \quad \qquad \ [Duchi algorithm]\\
%\STATE \qquad \quad end for;
%\STATE  \quad end if;
%\STATE  end for;
\STATE \textbf{Output:} membership matrix $X^{n+1} = [\mathbf x_1^{n+1}, \mathbf x_2^{n+1}, \dots, \mathbf x_N^{n+1}]$
\end{algorithmic}
\end{algorithm}
Similarly, the second value can be computed by noticing that $\langle X^\top X, S \rangle_F = \mbox{trace}\left(X^\top\big(X S\big)\right)$.
%\begin{align*}
%\langle X^\top X, S \rangle_F = \mbox{trace}\left(\big(X^\top X\big)^\top S\right) = \mbox{trace}\left(\big(X^\top X\big) S\right) = \mbox{trace}\left(X^\top\big(X S\big)\right).
%\end{align*}
The element $i$, $i = 1, \dots, N$, on the diagonal of the matrix $X^\top\big(X S\big)$ is
\begin{align*}
    [\mbox{row(i)} \  \mbox{of} \ X^\top] \times [\mbox{column(i)} \ \mbox{of} \ XS] =\langle \mathbf x_i, X \mathbf s_i \rangle.
\end{align*}
Thus, the trace of matrix $\left(X^\top\big(X S\big)\right)$ is the sum of $\langle \mathbf x_i, X \mathbf s_i \rangle$ 
%(the inner product of two vectors $\mathbf x_i$, $X \mathbf s_i$ used before to update the gradient) 
over all index $i = 1, 2, \dots, N$. Consequently, we have 
\begin{align*}%\label{loss2}
\langle X^\top X, S \rangle_F = \sum_{i = 1}^N \langle \mathbf x_i, X \mathbf s_i \rangle.
\end{align*}
It follows 
%from~\eqref{loss}, \eqref{loss1}, and \eqref{loss2} 
that
%\begin{equation*}
    $f(X) = \|S\|_F^2 + \|XX^\top\|_F^2 -2 \sum_{i = 1}^N \langle \mathbf x_i, X \mathbf s_i \rangle.$
%\end{equation*}
Updating the loss value at each iteration can now be done in parallel, as shown in Algorithm~\ref{GPA_para_loss}. Compared to the base algorithm, Algorithm~\ref{GPA_para_loss} makes use of vectors $\mathbf x_i, X \mathbf s_i$ needed to update the gradient to compute the inner product $\mbox{prod}_i$ in line 6. These inner product values are then merged in line 8 to compute the loss and evaluate it against the tolerance.%if it is good enough.

\smallskip

\noindent {\bf Nesterov Accelerated Gradient (FISTA).} Regarding the convergence rate of the objective function values $f(X_n)$,  
\begin{algorithm}
\caption{(Parallel FISTA for fuzzy clustering)}
\begin{algorithmic}[1] \label{fista_para_loss}
\STATE \textbf{Input:} similarity matrix $S \in \mathbb R^{N \times N}$,  cluster-num $C >0$, step size $\tau>0$;
\STATE initialize $\bar X^0 = [\bar{\mathbf x}_1^0, \bar{\mathbf x}_2^0, \dots, \bar{\mathbf x}_N^0] \in \Omega$, $X^1 := \bar X^0$, $\overline{\mbox{loss}}^{0}: = f(\bar X^0)$;
\STATE for each iteration $n = 1, \dots,$ max-iter do:\\
{\color{gray} \hspace{1em}//-----\textit{basic GPA}-----}
\STATE  \quad $\mbox{share}^{n} := X^n (X^n)^\top$; 
\STATE  \quad for $i := 1, \dots, N$ (in parallel) do:
\STATE  \qquad $\mbox{\bf grad}_i^{n}:= -4 (X^n \mathbf s_i - \mbox{share}^{n} \mathbf x_i^n)$;
\STATE \quad \quad $\bar{\mathbf x}_i^{n}:= P_{\Delta^C} (\mathbf x_i^n - \mbox{step size} \times \mbox{\bf grad}_i^{n})$;\\
%\STATE \quad end for;\\
{\color{gray} \hspace{1em}//-----\textit{check loss quality}-----}
 \STATE  \quad $\overline{\mbox{share}}^{n} := \bar X^n (\bar X^n)^\top$;
\STATE \quad for $i := 1, \dots, N$ (in parallel) do:
\STATE \quad \quad $\overline{\mbox{prod}}_i^{n}  := \quad \ \langle \bar X^n \mathbf s_i, \bar{\mathbf x}_i^n \rangle$ ;
%\STATE \quad end for;
\STATE  \quad $\overline{\mbox{merge}}^{n} := \overline{\mbox{prod}}_1^{n} + \overline{\mbox{prod}}_2^{n} + \dots + \overline{\mbox{prod}}_N^{n}$ ;
\STATE  \quad  $\overline{\mbox{loss}}^{n} \quad := \|S\|_F^2 + \|\overline{\mbox{share}}^{n}\|_F^2 - 2 \times \overline{\mbox{merge}}^{n}$ ;
\STATE  \quad if $\overline{\mbox{loss}}^{n} - \overline{\mbox{loss}}^{n-1} \leq \mbox{tol}$:
\STATE \quad \quad  \quad  break;\\
%\STATE  \quad  end if;\\
{\color{gray} \hspace{1em}//-----\textit{update inertial parameter and accelerated term}-----}
\STATE \quad 
 $t_{n+1} := \frac{1 + \sqrt{1 + 4t_n^2}}{2}$;
 \STATE  \quad for $i := 1, \dots, N$ (in parallel) do: 
 \STATE \qquad $\mathbf x_i^{n+1} := \bar{\mathbf x}_i^n + \left( \frac{t_n - 1}{t_{n+1}} \right)\left(\bar{\mathbf x}_i^n - \bar{\mathbf x}_i^{n-1}\right)$;
% \STATE \quad end for;
%\STATE end for;
\STATE \textbf{Output:} membership matrix $\bar X^{n} = [\bar{\mathbf x}_1^{n}, \bar{\mathbf x}_2^{n}, \dots, \bar{\mathbf x}_N^{n}]$
\end{algorithmic}
\end{algorithm}
the \textit{fast iterative shrinkage-thresholding algorithm} (FISTA, \cite{Beck_Teboulle_09}) -- a generalized version of the Nesterov accelerated gradient scheme \cite{Nesterov_83} for constrained convex optimization problems -- is well known to speed up the basic gradient projection method. Although there is no theoretical guarantee for nonconvex problems, we would like to examine the performance of the method for fuzzy clustering through numerical experiments.

\medskip
In FISTA, apart from the basic gradient projection step, we need to use not only the information from the current iteration but also from the previous iteration to define the next one using a special inertial term. This increases the computation cost and memory requirements per iteration. To initialize FISTA, we need points $\bar X^0 \in \Omega$, $X^1 := \bar X^0$, value $f(\bar X^0)$, step size $\tau > 0$, tolerance $\mbox{tol} > 0$, and set the inertial 
parameter $t_1 = 1$. Each iteration $n \geq 1$ starts with computing the gradient $\nabla f(X^n)$ and updating  the basic GPA
\begin{align*}%\label{fista_gpa}
         \bar X^n = P_{\Omega} \left[ X^n - \tau \nabla f(X^n) \right]
    \end{align*}
If the loss per iteration is not high enough, i.e., $f(\bar X^{n-1}) -  f(\bar X^n)\geq \mbox{tol}$, the inertial parameter and the accelerated term are updated with
    \begin{align*}
        t_{n+1} &= \frac{1 + \sqrt{1 + 4t_n^2}}{2} \\
        X^{n+1} &= \bar X^n + \left( \frac{t_n - 1}{t_{n+1}} \right) \left( \bar X^n - \bar X^{n-1} \right).
    \end{align*}
    
A parallel version of FISTA is described in Algorithm~\ref{fista_para_loss} where we follow the logic of parallelization of the basic GPA. 

%{\color{gray}\noindent{\bf Conjecture}: By the  numerical performance, can one derive a ``local" convergence property of FISTA for this problem?}

\section{Refining the set of critical points}
Recall that due to the non-convexity of the problem ($\mathcal P$), GPA can guarantee at most a critical point, an $\bar X \in \Omega$ satisfying the first-order optimality condition~\eqref{opt_con}. There could be many such points, and thus it is ideal to refine the set of critical points into a smaller set containing local solutions. Aiming at this, we will provide second-order optimality conditions for ($\mathcal P$).

\medskip

As the constraint set $\Omega$ is convex, the first-order optimality condition~\eqref{opt_con} can be rewritten (\cite[Lemma~3.13 and Theorem~3.24]{Rusz_06}) as 
\begin{align}\label{1st_cond}
    \langle \nabla f(\bar X), V \rangle_F \geq 0, \quad  \forall V \in T(\bar X; \Omega)
\end{align}
where $T(\bar X; \Omega)$ stands for the \textit{Bouligand-Severi tangent cone} to $\Omega$ at $\bar X \in \Omega$.   A direction $V \in \mathbb R^{C \times N}$ is called (\cite[Definition~3.11]{Rusz_06}) a \textit{tangent direction} to the set $\Omega$ at $\bar X \in \Omega$,  $V \in T(\bar X; \Omega)$, if there exist sequences of points $X^n \in \Omega$ and scalars $\tau_n > 0$, $n = 1, 2, \dots,$ such that $\tau_n \downarrow 0$ and 
\begin{equation*}
    V = \lim_{n \to \infty} \dfrac{X^n - \bar X}{\tau_n}.
\end{equation*} The concept of tangent direction is fundamental in analyzing perturbations around $\bar X$ -- a candidate for minimizers of the optimization problem -- via points in $\Omega$ converging to $\bar X$.

\medskip
 The gradient projection method updates by stepping in the negative gradient direction and projecting onto the constraint set $\Omega$. In~\eqref{1st_cond}, if a direction $V \in T(\bar{X}; \Omega)$ satisfies $\langle \nabla f(\bar{X}), V \rangle_F = 0$, the directional derivative of $f$ at $\bar{X}$ along $V$ vanishes, indicating that $f$ does not decrease in that direction. As a result, the method may get stuck at $\bar{X}$. In such cases, we need to exploit further second-order approximations, the Hessian of the objective function and the second-order tangent cone of the constraint set, to confirm if $\bar{X}$ is a local minimizer.

\medskip
 Let $\bar X \in \Omega$ and $\bar V \in T(\bar X; \Omega)$. One calls (\cite[Definition~3.41]{Rusz_06}) $W \in \mathbb R^{C \times N}$ a \textit{second-order tangent direction} to the set $\Omega$ at the point $\bar X$ in direction $\bar V$ if there exist sequences of points $X^n \in \Omega$ and scalars $\tau_n > 0$, $n = 1, 2, \dots,$ such that $\tau_n \downarrow 0$ and 
\begin{equation*}
    W = \lim_{n \to \infty} \dfrac{X^n - \bar X - \tau_n\bar V}{\frac{1}{2}\tau_n^2}.
\end{equation*} The set of second-order tangent directions to $\Omega$ at the point $\bar X$ in direction $\bar V$ is denoted by~$T^2(\bar X, \bar V; \Omega)$.

\medskip
The next theorem provides us with second-order necessary optimality conditions for ($\mathcal P$). These conditions pair second-order information of the objective function (Hessian) with first-order information of the constraint set (tangent cone) and vice versa,  first-order information of the objective function (gradient) with second-order information of the constraint set (second-order tangent cone).

\begin{theorem}\label{2nd_con_thm}
Let $\bar X$ be a local minimizer of ($\mathcal P$). In addition, let $\bar V \in \mathbb R^{C\times N}$ be a direction such that $\langle\nabla f(\bar X), \bar V \rangle_F = 0$ and $\bar V \in T(\bar X; \Omega)$. Then
\begin{align}\label{2nd_con_a}
 \langle \nabla^2 f(\bar X) \bar{V}, \bar{V} \rangle_F \geq 0  
\end{align} 
and 
\begin{align}\label{2nd_con_b}
 \langle \nabla f(\bar{X}), W \rangle_F \geq 0, \quad \forall W \in T^2(\bar{X}, \bar{V}; \Omega),
\end{align}
where the Hessian $\nabla^2 f(\bar X)$ 
%of $f$ at $\bar X$, as a linear operator, 
maps $V$ % \in \mathbb{R}^{C \times N}$ 
to
\begin{equation}\label{hessian}
\nabla^2 f(\bar X)V = -4 V (S - \bar X^T \bar X) + 4 \bar X (V^T \bar X + \bar X^T V) \in \mathbb{R}^{C \times N}
\end{equation} and the second-order tangent cone to $\Omega$ at $\bar X$ in direction $\bar V$ is given by
\begin{align}\label{2nd_tangent_cone}
\!\!T^2(\bar X,\! \bar V; \Omega)\! = \left\{ W\! =\! (w_{ki}) \in \mathbb{R}^{C\! \times\! N} \middle| \sum_{k=1}^C w_{ki}\! =\! 0,  \forall i;  w_{ki} \geq 0 \text{ if} \begin{cases}
 \bar{x}_{ki}\! =\! 0\\    
 \bar{v}_{ki}\! =\! 0
\end{cases}\!\!\right\}.
\end{align}
\end{theorem}
\begin{proof}
Since the objective function is $C^2$-smooth and the constraint set is a polyhedral, applying~\cite[Theorem~3]{An_Yen_21} we get~\eqref{2nd_con_a} and~\eqref{2nd_con_b}. It remains to show formulas~\eqref{hessian} and~\eqref{2nd_tangent_cone}.

\medskip

Given $\bar X \in \mathbb R^{C \times N}$, the Hessian $\nabla^2 f(\bar X)$ is the Fr\'echet derivative of the gradient operator $X \mapsto \nabla f(X) = -4X(S - X^TX)$ at $\bar X$. First, it is not difficult to verify that the operator $$V \mapsto \nabla^2 f(\bar X)V:=-4 V (S - \bar X^T \bar X) + 4 \bar X (V^T \bar X + \bar X^T V)$$ is a continuous linear operator from $\mathbb R^{C \times N}$ to $\mathbb R^{C \times N}$. We next justify that 
\begin{equation}\label{hessian_lim}
\lim_{\|V\|_F \to 0} \frac{\|\nabla f(\bar X + V) - \nabla f(\bar X) - \nabla^2 f(\bar X)V\|_F}{\|V\|_F} = 0
\end{equation}
with $\nabla^2 f(\bar X)V$ given in~\eqref{hessian}.

\medskip
From the formula for the gradient, we have \begin{align*}
\nabla f(\bar X + V) 
%&= -4(\bar X + V)\left(S - (\bar X + V)^T (\bar X + V) \right) \\
%&= -4(\bar X + V)\left(S - \bar X^T \bar X - V^T \bar X - \bar X^T V - V^T V\right) \\
&= -4(\bar X + V)(S - \bar X^T \bar X) 
+ 4(\bar X + V)(V^T \bar X + \bar X^T V + V^T V).
\end{align*}
Thus,
\begin{align*}
\nabla f(\bar X + V) - \nabla f(\bar X) 
&= -4 V(S - \bar X^T \bar X) + 4 \bar X (V^T \bar X + \bar X^T V) \\
&\quad + 4 V (V^T \bar X + \bar X^T V) + 4 (\bar X + V) V^T V.
\end{align*}
Now, using~\eqref{hessian} we have
\begin{align*}
\nabla f(\bar X + V) - \nabla f(\bar X) - \nabla^2 f(\bar X)V
%&= -4 V(S - \bar X^T \bar X) + 4 \bar X (V^T \bar X + \bar X^T V) \\
%&\quad + 4 V (V^T \bar X + \bar X^T V) + 4 (\bar X + V) V^T V\\
%& \quad+4 V (S - \bar X^T \bar X) - 4 \bar X (V^T \bar X + \bar X^T V)\\
& = 4 V (V^T \bar X + \bar X^T V) + 4 (\bar X + V) V^T V.
\end{align*}
Hence, bounding the norm yields
\begin{align*}
\|\nabla f(\bar X\! +\! V)\! -\! \nabla f(\bar X)\! -\! \nabla^2\! f(\bar X)V\|_F & =\|4 V (V^T \bar X + \bar X^T V) + 4 (\bar X + V) V^T V \|_F \\
& \leq 4( \|V (V^T\! \bar X\! +\! \bar X^T\! V) \|_F\! +\! \|(\bar X\! +\! V) V^T\! V \|_F)\\
& \leq 4 (2 \|V\|_F^2 \|\bar X \|_F + (\|\bar X\|_F + \|V \|_F) \| V\|_F^2).
\end{align*}
It follows that
\begin{align*}
 \lim_{\|V\|_F \to 0} & \frac{\|\nabla f(\bar X + V) - \nabla f(\bar X) - \nabla^2 f(\bar X)V\|_F}{\|V\|_F} \\ &= \lim_{\|V\|_F \to 0} 4 (2 \|V\|_F \|\bar X \|_F + (\|\bar X\|_F + \|V \|_F) \| V\|_F)= 0,
\end{align*}
justifying~\eqref{hessian_lim}. 

\medskip
We now prove~\eqref{2nd_tangent_cone}. By the  polyhedral structure,  
\begin{equation*}
    \Omega = \left\{X = (x_{ki}) \in \mathbb{R}^{C \times N} \ \middle| \ \sum_{k = 1}^C x_{ki} = 1, \ \forall i, \ \text{and} \ x_{ki} \geq 0, \ \forall k, i \right\},
\end{equation*} the tangent cone to $\Omega$ at $\bar{X} \in \Omega$ can be given explicitly as 
\begin{equation}\label{1st_tangent_cone}
T(\bar X; \Omega) = \left\{ V = (v_{ki}) \in \mathbb{R}^{C\! \times\! N} \middle|  \sum_{k=1}^C v_{ki} = 0, \ \forall i, \, \mbox{and } v_{ki} \geq 0 \text{ if } \bar{x}_{ki} = 0\right\};
\end{equation} see, e.g., \cite[formula~(3.13)]{Ban_etal_11}. Since the latter is also a polyhedral, the second-order tangent cone $T^2(\bar X, \bar V; \Omega)$ can be computed (\cite[Lemma 3.43]{Rusz_06}) via the (first-order) tangent cone as $T^2(\bar X, \bar V; \Omega) = T\left(\bar V; T(\bar X; \Omega)\right)$. Thus, invoking~\eqref{1st_tangent_cone} we obtain~\eqref{2nd_tangent_cone}, 
which completes the proof.
\end{proof}

\begin{remark}{\rm 
As pointed out in~\cite{An_Yen_21}, conditions~\eqref{2nd_con_a} and~\eqref{2nd_con_b} together form a stronger version of the second-order necessary optimality condition stated in \cite[Theorem~3.45] {Rusz_06}: \textit{If $\bar X \in \Omega$ is a local minimizer, then for every $\bar V \in T(\bar X; \Omega)$ with $\langle\nabla f(\bar X), \bar V \rangle_F = 0$ one has $$\langle \nabla^2 f(\bar X) \bar{V}, \bar{V} \rangle_F  + \langle \nabla f(\bar{X}), W \rangle_F\geq 0, \quad \forall W \in T^2(\bar X, \bar V; \Omega).$$} This, in particular, makes the procedure of refining the set of critical points easier by showing separately that if either ~\eqref{2nd_con_a} or~\eqref{2nd_con_b} is violated at $\bar X$, then $\bar X$ is \textit{not} a local solution of~($\mathcal P$).}
\end{remark}

\begin{remark}{\rm 
Condition~\eqref{2nd_con_a} is satisfied at $\bar X \in \Omega$ for all direction $\bar V \in T(\bar X; \Omega)$ and $\langle\nabla f(\bar X), \bar V \rangle_F = 0$ if the optimal value $\mu_1$ of the minimization problem
\begin{equation}\label{subprob1}
    \displaystyle\min_{V \in \mathbb R^{C \times N}} \left\{\langle \nabla^2 f(\bar X) V, V \rangle_F \mid V \in T(\bar X; \Omega) \; \& \;\langle\nabla f(\bar X), V \rangle_F = 0  \right\} \tag{$\mathcal P_1$}
\end{equation}
is non-negative. For a fixed $\bar X \in \Omega$, the constraint set of~\eqref{subprob1} is a polyhedron, while the objective function is of quadratic form. Solving nonconvex quadratic programs is NP-hard in general, which requires employing global optimization techniques like branch-and-cut using spatial-branching, as it is available, for example, in solvers such as
SCIP\footnote{\href{https://scopopt.org}{https://scipopt.org}}~\cite{SCIP8},
BARON\footnote{\href{https://minlp.com/}{https://minlp.com}}, Gurobi\footnote{\href{https://www.gurobi.com/events/non-convex-quadratic-optimization/}{https://www.gurobi.com}}, or FICO Xpress \footnote{\href{https://www.fico.com/en/products/fico-xpress-optimization}{https://www.fico.com}}. 

Similarly, condition~\eqref{2nd_con_b} is fulfilled at a $\bar X \in \Omega$ if the optimal value $\mu_2$ of the linear program
\begin{equation}\label{subprob2}
   \! \displaystyle\min_{W, V \in \mathbb R^{C \times N}}\! \left\{\langle \nabla f(\bar{X}), W \rangle_F \mid W \in T^2(\bar{X}, V; \Omega), V \in T(\bar X; \Omega), \langle\nabla f(\bar X), V \rangle_F = 0  \right\} \tag{$\mathcal P_2$}
\end{equation}
is non-negative. Compared to~\eqref{subprob1}, solvers for~\eqref{subprob2} are more available even at large scales, for example, additionally to the above, also HiGHS\footnote{\href{https://highs.dev/}{https://highs.dev}}, COPT\footnote{\href{https://www.copt.de}{https://www.copt.de}}, and CPLEX\footnote{\href{https://www.ibm.com/products/ilog-cplex-optimization-studio}{https://www.ibm.com}}. Note also from the formulas of the tangent cones~\eqref{2nd_tangent_cone} and~\eqref{1st_tangent_cone} that $$T^2(\bar{X}, V; \Omega) =  T(\bar X; \Omega), \quad \forall V \in T(\bar X; \Omega)$$ when $\bar X = (x_{ki})$ is a critical point in the interior of $\Omega$, $x_{ki} > 0$ for all $k, i$. Thus, condition~\eqref{2nd_con_b} is already satisfied at $\bar X$ by the first-order optimality condition~\eqref{1st_cond} and solving~\eqref{subprob2} is not needed.  
}
\end{remark}
We conclude the section with a toy example where GPA identifies multiple critical points, highlighting its strengths and the application of Theorem~\ref{2nd_con_thm} in refining them.

\begin{ex}\label{ex_7node_graph}{\rm 
Consider the fuzzy clustering problem \((\mathcal{P})\) for a 7-node graph, illustrated in Fig.~\ref{7n2c-pic}. The graph is designed to mimic a small citation network, where each node represents a paper. The nodes are organized into two groups: Group 1, consisting of papers \(1, 2, 3, 4\), and Group 2, consisting of papers \(4, 5, 6, 7\). Paper \(4\) serves as a bridge between the two groups -- it cites papers \(2\) and \(3\) (from Group 1), and is in turn cited by papers \(5\) and \(6\) (from Group 2). The similarity matrix \(S\) is constructed from the graph's adjacency matrix by adding $1$s along the diagonal. This reflects the assumption that papers are ``similar'' if one cites the other, i.e., \(s_{ij} = 1\) if paper \(i\) cites or is cited by paper \(j\), and each paper is considered ``most similar'' to itself with \(s_{ii} = 1\).

\begin{figure}[h!]
  \centering
  % Subfigure A: Graph
  \begin{subfigure}[b]{0.45\textwidth}
    \centering
    \resizebox{0.75\linewidth}{!}{
    \begin{tikzpicture}[every node/.style={circle, draw=black, minimum size=0.5cm, inner sep=0pt}, thick]
      % Nodes
      \node (4) at (0, 0) {4};
      \node (2) at (-1, 1) {2};
      \node (3) at (1, 1) {3};
      \node (1) at (-2, 2) {1};
      \node (5) at (-1, -1) {5};
      \node (6) at (1, -1) {6};
      \node (7) at (2, -2) {7};

      % Edges
      \draw (1) -- (2);
      \draw (2) -- (3);
      \draw (2) -- (4);
      \draw (3) -- (4);
      \draw (4) -- (5);
      \draw (4) -- (6);
      \draw (5) -- (6);
      \draw (6) -- (7);
    \end{tikzpicture}
    }
    \caption{Graph with $7$ nodes}
  \end{subfigure}
  \hfill
  % Subfigure B: Similarity matrix
  \begin{subfigure}[b]{0.45\textwidth}
    \centering
    \[
    S = \begin{bmatrix}
    1 & 1 & 0 & 0 & 0 & 0 & 0 \\
    1 & 1 & 1 & 1 & 0 & 0 & 0 \\
    0 & 1 & 1 & 1 & 0 & 0 & 0 \\
    0 & 1 & 1 & 1 & 1 & 1 & 0 \\
    0 & 0 & 0 & 1 & 1 & 1 & 0 \\
    0 & 0 & 0 & 1 & 1 & 1 & 1 \\
    0 & 0 & 0 & 0 & 0 & 1 & 1
    \end{bmatrix}
    \]
    \caption{Similarity matrix $S$}
  \end{subfigure}
  \caption{A $7$-nodes graph and its similarity matrix $S$}\label{7n2c-pic}
\end{figure}

\medskip
Given that the graph naturally forms two loosely connected clusters with node \(4\) in common, we set the number of clusters to \(C = 2\).  We expect GPA to produce a membership matrix \(\bar{X} \in \mathbb{R}^{2 \times 7}\) that accurately captures this structure. Algorithm~\ref{GPA_para} was implemented using Jupyter Notebook on a MacBook with a 1.4~GHz Quad-Core Intel Core~i5 processor, a step size of \(0.1\), and was terminated when the objective function values ceased to decrease. We obtained three different outputs depending on the choice of starting points; see Table~\ref{7n2c_result_table}.

\medskip
 In Scenario 1, where the starting point was initialized randomly, the objective function decreased to a value of \(6.49\). The resulting membership matrix \(\bar{X}_1 = (\bar{x}_{ki})\) closely aligns with the expected structure: papers \(2\) and \(6\) strongly define the two clusters (\(\bar{x}_{12} = 1\), \(\bar{x}_{22} = 0\); \(\bar{x}_{16} = 0\), \(\bar{x}_{26} = 1\)), paper \(4\) participates in both clusters (\(\bar{x}_{14}, \bar{x}_{24} \approx 0.5\)), papers \(1\) and \(3\) are mainly associated with Cluster 1 (\(\bar{x}_{11}, \bar{x}_{13} \approx 1\)), and papers \(5\) and \(7\) are primarily associated with Cluster 2 (\(\bar{x}_{25}, \bar{x}_{27} \approx 1\)). We repeated the experiment multiple times using different randomly generated initial points and found the results to be stable. We conclude that GPA successfully recovers the two underlying clusters in this graph instance.

 \medskip  
In Scenario 2, the starting point was biased: all seven nodes were assigned to Cluster~1, and none to Cluster~2 (i.e., row~$1$ was filled with $1$s, and row~$2$ with $0$s). The objective function decreased to a value of \(8.84\). The resulting membership matrix \(\bar{X}_2 = (\bar{x}_{ki})\) exhibits a different structure: paper~$4$ anchors Cluster~$1$, grouping together papers~$2$, $3$, $5$, and~$6$, all of which are directly connected via citations. Meanwhile, papers~$1$ and~$7$ comprise Cluster~$2$. Although this solution yields a higher objective function value than in Scenario~1, it remains plausible -- effectively grouping strongly connected papers while separating those with weaker connections.
\begin{table}[ht!]
\begin{center}
\renewcommand{\arraystretch}{2}
\begin{tabular}{>{\raggedright\arraybackslash\small}m{7.5cm}>{\centering\arraybackslash}m{3.75cm}}
\hline
\textbf{Outputs by GPA for the $7$-node graph} & \textbf{Illustration}  \\
\hline
\noindent{\bf Scenario 1.} The starting point is initialized randomly, the loss is at $6.49$ with $\bar X_1:$ 

\bigskip
$\begin{bmatrix}
0.8835 & \bf{1.} & 0.9096 & \bf{0.5202} & 0.1163 & \bf{0.} & 0.0906 \\
0.1165 & \bf{0.} & 0.0904 & \bf{0.4798} & 0.8837 & \bf{1.} & 0.9094
\end{bmatrix}$
& \medskip
\begin{tikzpicture}[thick,scale=0.7, every node/.style={circle, draw=black, minimum size=0.5cm, inner sep=0pt}]

  % Nodes with dominant cluster color
  \node[fill=blue!30] (1) at (-2, 2) {1};
  \node[fill=blue!60] (2) at (-1, 1) {2};
  \node[fill=blue!40] (3) at (1, 1) {3};
  \node[fill=blue!25!red!25] (4) at (0, 0) {4}; % mixed, color blend
  \node[fill=red!40] (5) at (-1, -1) {5};
  \node[fill=red!60] (6) at (1, -1) {6};
  \node[fill=red!30] (7) at (2, -2) {7};

  % Edges
  \draw (1) -- (2);
  \draw (2) -- (3);
  \draw (2) -- (4);
  \draw (3) -- (4);
  \draw (4) -- (5);
  \draw (4) -- (6);
  \draw (5) -- (6);
  \draw (6) -- (7);
  
  %Clusters
  % Cluster 1 
  \draw[blue, thick, dashed, rounded corners=12pt] 
    (-2.5, -0.5) rectangle (1.5, 2.5);
  % Cluster 2 
  \draw[red, thick, dashed, rounded corners=12pt] 
    (-1.8, -2.5) rectangle (2.5, 0.5);
    
\end{tikzpicture} \\
\hline
\noindent{\bf Scenario 2.} The starting point has $1$s in row $1$ and $0$s in row $2$; the loss is at~$8.84$~with~$\bar X_2:$

\bigskip
$\begin{bmatrix}
0.1308 & 0.6435 & 0.8692 & \bf{1.} & 0.8692 & 0.6435 & 0.1308 \\
0.8692 & 0.3565 & 0.1308 & \bf{0.} & 0.1308 & 0.3565 & 0.8692
\end{bmatrix}$ 
& \medskip
\quad \begin{tikzpicture}[thick,scale=0.7,  every node/.style={circle, draw=black, minimum size=0.5cm, inner sep=0pt}]
  % Nodes (colored by dominant membership in \bar{X}_2)
  \node[fill=red!40] (1) at (-2, 2) {1};
  \node[fill=blue!25] (2) at (-1, 1) {2};
  \node[fill=blue!40] (3) at (1, 1) {3};
  \node[fill=blue!60] (4) at (0, 0) {4};
  \node[fill=blue!40] (5) at (-1, -1) {5};
  \node[fill=blue!25] (6) at (1, -1) {6};
  \node[fill=red!40] (7) at (2, -2) {7};

  % Edges
  \draw (1) -- (2);
  \draw (2) -- (3);
  \draw (2) -- (4);
  \draw (3) -- (4);
  \draw (4) -- (5);
  \draw (4) -- (6);
  \draw (5) -- (6);
  \draw (6) -- (7);
  
  % Clusters
  % Cluster 1 (center group): blue rounded box
  \draw[blue, thick, dashed, rounded corners=10pt]
    (-1.4, -1.4) rectangle (1.4, 1.4);

   % Cluster 2 
  \draw[red, thick, dashed, rounded corners=10pt]
  (1.6, 1.6)  -- (1.6, -2.5) -- (2.5, -2.5) -- (2.5, 2.5) --(-2.5, 2.5) -- (-2.5, 1.6) -- (1.6, 1.6);

\end{tikzpicture} \\
\hline
\noindent{\bf Scenario 3.} The starting point is the matrix with $0.5$s in each element; the loss is at~$12.25$~with~$\bar X_3:$

\bigskip
$\begin{bmatrix}
0.5 & 0.5 & 0.5 & 0.5 & 0.5 & 0.5 & 0.5 \\
0.5 & 0.5 & 0.5 & 0.5 & 0.5 & 0.5 & 0.5
\end{bmatrix}$
& \medskip
\begin{tikzpicture}[thick,scale=0.7, every node/.style={circle, draw=black, minimum size=0.5cm, inner sep=0pt}]
  %Nodes
  % No dominant clusters, color blend
  \node[fill=blue!25!red!25] (1) at (-2, 2) {1};
  \node[fill=blue!25!red!25] (2) at (-1, 1) {2};
  \node[fill=blue!25!red!25] (3) at (1, 1) {3};
  \node[fill=blue!25!red!25] (4) at (0, 0) {4}; 
  \node[fill=blue!25!red!25] (5) at (-1, -1) {5};
  \node[fill=blue!25!red!25] (6) at (1, -1) {6};
  \node[fill=blue!25!red!25] (7) at (2, -2) {7};

  % Edges
  \draw (1) -- (2);
  \draw (2) -- (3);
  \draw (2) -- (4);
  \draw (3) -- (4);
  \draw (4) -- (5);
  \draw (4) -- (6);
  \draw (5) -- (6);
  \draw (6) -- (7);

  % Clusters
  % Cluster 1 
  \draw[blue, thick, dashed, rounded corners=12pt] 
    (-2.5, -2.5) rectangle (2.5, 2.5);
  % Cluster 2 
  \draw[red, thick, dashed, rounded corners=12pt] 
    (-2.3, -2.3) rectangle (2.7, 2.7);
\end{tikzpicture} \\
\hline
\end{tabular}
\end{center}
\caption{Outputs by GPA with different starting points for the $7$-node graph}
\label{7n2c_result_table}
\end{table}

\medskip  
In Scenario 3, the starting point was uniformly initialized with \(0.5\) in all entries, implying that each paper belongs equally to both clusters. The algorithm terminated after the first iteration, yielding a membership matrix \(\bar{X}_3\) identical to the initial point, with an objective function value of \(12.25\). Though being a critical point, $\bar X_3$ provides no meaningful insight into the graph’s structure.

\medskip  
It turns out that \(\bar{X}_3\) is \textit{not} a local solution to \((\mathcal{P})\), as it violates the second-order necessary optimality condition~\eqref{2nd_con_a}:  
\[
\langle \nabla^2 f(\bar{X}_3)\, \bar{V}, \bar{V} \rangle_F = -4 < 0
\]  
with  
\[
\bar{V} = \begin{bmatrix}
1 & 0 & 0 & 0 & 0 & 0 & 0 \\
-1 & 0 & 0 & 0 & 0 & 0 & 0
\end{bmatrix} \in T(\bar{X}_3; \Omega) \ \mbox{ satisfying } \ \langle \nabla f(\bar{X}_3), \bar V \rangle_F = 0.
\]  

The remaining candidates for local solutions are $\bar{X}_1$ and $\bar{X}_2$. Since $\bar{X}_2$ yields a higher objective function value than $\bar{X}_1$, it cannot be a global solution. Numerical results from the Gurobi solver confirm that $\bar{X}_1$ is a global optimum.

\medskip
As a side note, the Louvain and Leiden algorithms (implemented using the \texttt{igraph} and \texttt{leidenalg} libraries in Python) returned two disjoint clusters, $\{1, 2, 3\}$ and $\{4, 5, 6, 7\}$, which capture less structural information from the graph than the GPA solution $\bar{X}_1$.

}
\end{ex}
\section{Experiments with medium and large datasets}
In this section, we examine the performance of the parallel GPA and parallel FISTA algorithms on medium- to large-scale instances. We consider two datasets built from real citation networks: 
\begin{itemize} 
\item a medium-sized instance consisting of approx. $700$k articles linked by $4.6$ million citations, derived from the Web of Science; 
\item a large-sized instance comprising about $60$ million articles linked by $1.2$ billion citations, based on the OpenAlex dataset. 
\end{itemize} Due to their real and complex graph structures, these instances provide a suitable setting to observe how the parallel GPA and parallel FISTA improve objective function values over iterations.
Additionally, we include a synthetic medium-sized instance with two known clusters, enabling us to evaluate the quality of solutions produced by the two algorithms. Details on data preparation are provided below.

\medskip
\noindent{\textbf{WoS citation subgraph.}} We utilized the Web of Science (WoS) database \cite{webofscience}, supported by the German Competence Network for Bibliometrics, and extracted a subgraph of papers in two subjects: \textit{Mathematics} (M) and \textit{Operations Research \& Management Science} (OR\&MS) after cleaning. Nodes correspond to scholarly works and directed edges to citation relationships. Cleaning excluded non-English papers, those missing key metadata (year, author, WoS unique identifier, title, journal), papers beyond the range 2000--2024, and invalid references, resulting in 964,811 nodes and 5,087,058 edges. After extracting the largest connected component (796,467 nodes, 4,930,134 edges) and iteratively removing degree-1 nodes, we refined the graph to 722,623 nodes and 4,856,290 edges. For our future embedding tasks, we removed papers lacking abstracts and reprocessed the graph, yielding a final version with 698,135 nodes and 4,590,190 edges. %The graph is stored in \texttt{abstr\_WOS\_2subj\_2000-2024.gph}, where the first line lists the number of nodes and edges, followed by lines specifying each edge as \texttt{outnodeID innodeID}.

\medskip
\noindent{\textbf{OpenAlex connected citation graph.}} The OpenAlex connected citation graph represents the largest connected component of the OpenAlex citation dataset \cite{openalex} after cleaning and preprocessing.  The dataset, based on the OpenAlex snapshot (July 31, 2024), was cleaned by excluding works without journal, authors, or title, with years outside 1901–2024, non-English language, non-article/chapter types, or marked as retracted. Invalid references were also removed. After building the graph and extracting its largest connected component using a BFS-like algorithm, the final graph after degree-1 removal contains 59,343,462 nodes, 1,176,978,458 edges.% as described in the file \texttt{OA\_connected.gph} with the same format as the file \texttt{abstr\_WOS\_2subj\_2000-2024.gph}.

\medskip
\noindent{\textbf{Artificial graph.}} An artificial graph was constructed to mirror the scale of the WoS citation subgraph using a two-cluster Erdős–Rényi model. Two random graphs were generated separately: Cluster 1 (500,000 nodes, edge probability $p_{c1} = 1200/499999$) and Cluster 2 (250,000 nodes, edge probability $p_{c2} = 1600/249999$), with Cluster 2 nodes re-indexed to avoid ID overlaps before merging. A fixed number $200$ of inter-cluster edges was then added by randomly connecting node pairs between the two clusters to control the level of inter-connectivity. After merging, nodes with degree smaller than $1$ were iteratively pruned, and the final graph was verified to form a single connected component. %Node IDs were relabeled consecutively starting from 1.  The graph is recorded in the file \texttt{artificial\_750000\_437499673\_200\_cluster.gph} with the same format as the above two graphs. 

\medskip
\noindent{\bf Hardware.} All of the experiments in this section were conducted on a high-performance computing system equipped with two Intel® Xeon®  Gold $6132$ CPUs, providing a total of $28$ cores and $56$ threads at $2.60$ GHz base clock speed ($3.7$ GHz max). The system includes $376$ GiB of RAM and four NVIDIA Tesla V$100$-SXM$2$-$16$GB GPUs with $5.120$ CUDA cores and $16$GB HBM$2$ memory. The setup ran on CUDA $12.2$ with NVIDIA driver $535.183.01$. 

\medskip
\noindent{\bf Hyperparameters.} For every graph, the number of clusters $C$ was set to $2$ and the similarity matrix $S$ was formed by adding ones on the diagonal of the adjacency matrix, as done with the $7$-node tested graph in~Example~\ref{ex_7node_graph}. 

\medskip
To get the algorithms running, it remains to initialize a membership matrix $X^0 \in \Omega$ and to choose a step size $\tau$.  Theorem~\ref{gpa_conv_thm} states that the objective function values are non-increasing over iterations of the parallel GPA if the constant step size $\tau$ is small enough, regardless of the choice of $X^0 \in \Omega$. Thus, unless otherwise stated, we use a randomly generated $X^0 \in \Omega$. Shown in the proof of the theorem, a viable step size is  $$\tau \leq \bar\tau =  \dfrac{1}{4\|S\|_2 + 12N}$$ with \( \| S \|_2 \) the spectral norm of \( S \). As computing the spectral norm of \( S \) can be expensive with large datasets, we replace $\|S\|_2$ with $\|S\|_F$. The latter is easy to compute in our situation, by counting the number of non-zeros in $S$, as $S$ is obtained from the adjacency matrix of the graph with $1$s on the diagonal. A sufficiently small step size for the medium-size instance is $10^{-8}$, and for the large-size instance $10^{-10}$. 
We stopped the computation when there was no further decrease in the loss.

\medskip
\noindent{\bf Experiment 1: How parallel GPA and FISTA can reduce the loss?} \\
\begin{figure}[h!]
  \centering
  \includegraphics[width=\linewidth]{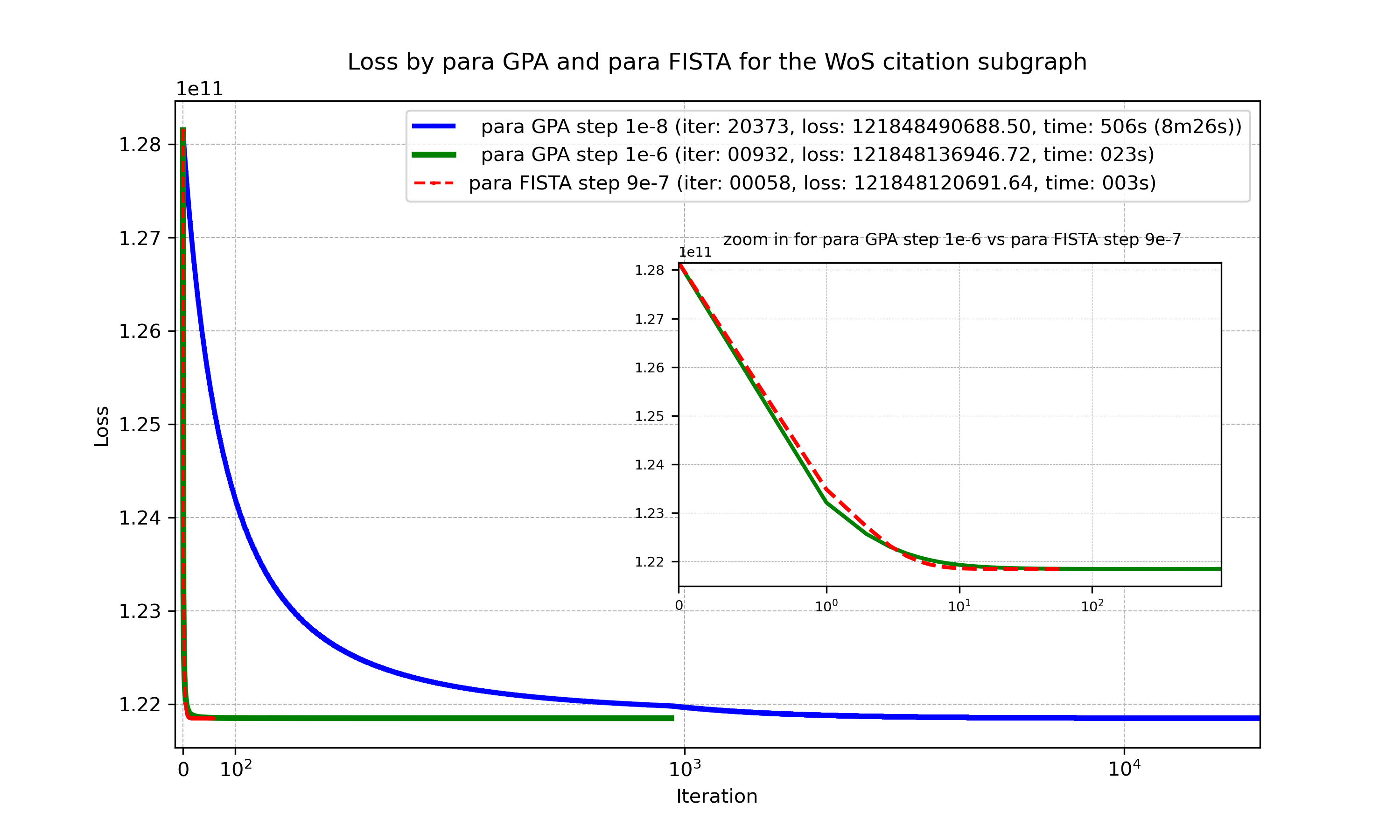}
  \caption{Loss over iterations by parallel GPA and FISTA for the WoS citation subgraph with different step sizes}
  \label{loss_gpa_fista_wos700k_pic}
\end{figure}
In this experiment, we evaluate the effectiveness of parallel GPA and FISTA in reducing the objective loss. 

\medskip
For the WoS medium-sized instance, we first test the theoretically viable step size $10^{-8}$. As expected from Theorem~\ref{gpa_conv_thm}, parallel GPA with this step size results in a monotonically decreasing loss over $20{,}000$ iterations, taking about $8.5$ minutes. However, empirical tuning reveals that the algorithm can tolerate much larger step sizes. For example, with a step size of $10^{-6}$, a factor of $100$ larger, parallel GPA converges to a comparable (and actually lower) loss in fewer than $1{,}000$ iterations and just $23$ seconds.

\medskip
This shows that while the theoretical step size ensures convergence, practical performance can be significantly improved by using a well-chosen, larger step size. Nevertheless, even with the tuned step size, GPA may still require many iterations with diminishing returns in loss reduction. This motivates us to consider parallel FISTA as a heuristic acceleration technique -- despite its lack of convergence guarantees in our non-convex setting. With a slightly smaller step size of $9\times 10^{-7}$, parallel FISTA reaches the best loss achieved by parallel GPA in only $58$ iterations and $3$ seconds -- demonstrating a substantial speedup. A detailed comparison of the performance of parallel GPA and FISTA is shown in Fig.~\ref{loss_gpa_fista_wos700k_pic}.

\medskip
\begin{figure}[h!]
  \centering
  \includegraphics[width=\linewidth]{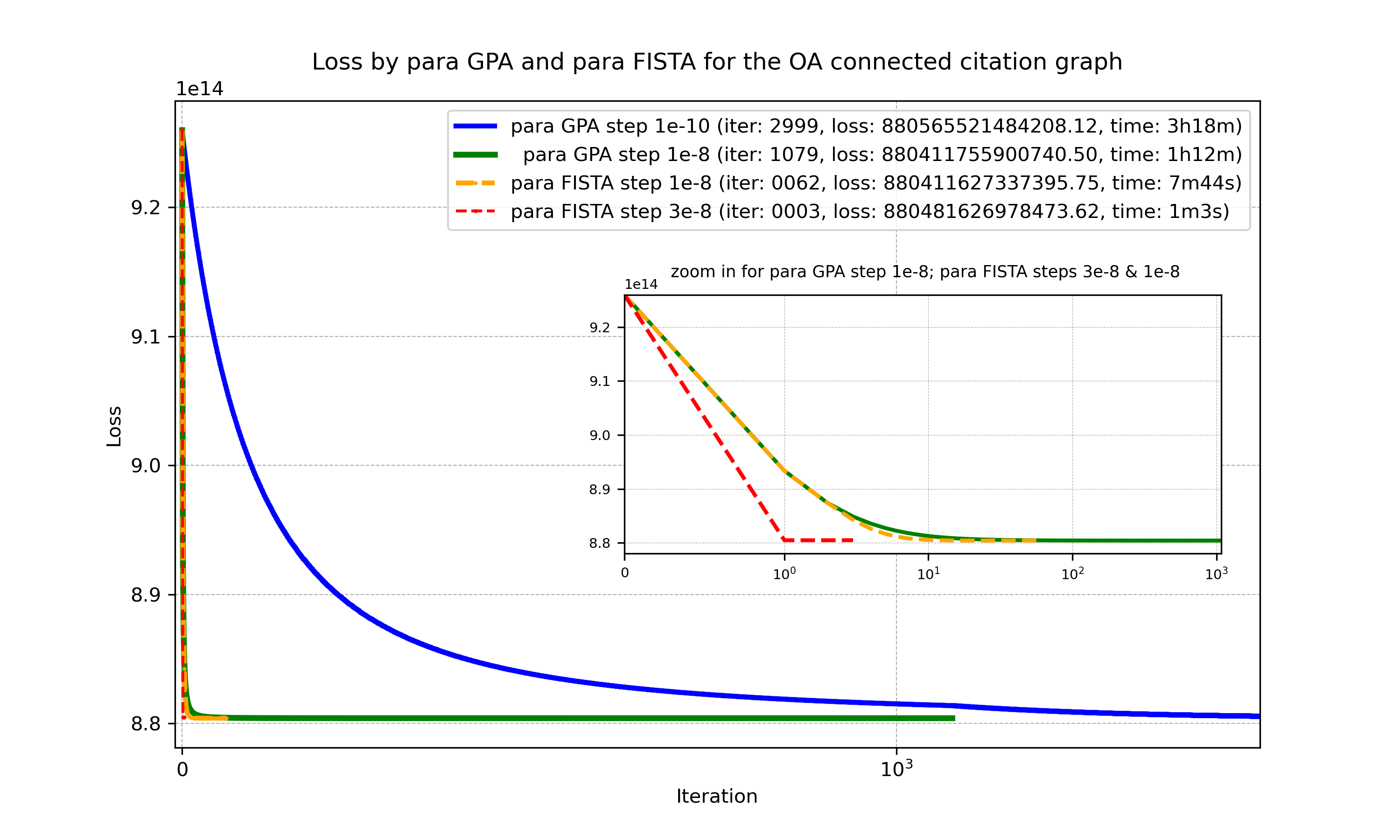}
  \caption{Loss over iterations by parallel GPA and FISTA for the OpenAlex connected citation graph with different step sizes}
  \label{loss_gpa_fista_oa_connected_pic}
\end{figure}
We now evaluate the performance of parallel GPA and FISTA on the large-scale OpenAlex connected citation graph. As with the WoS instance, we begin by applying a conservative step size of $10^{-10}$, which is well below the theoretical threshold established in Theorem~\ref{gpa_conv_thm}. As guaranteed by the theorem, this ensures that parallel GPA achieves a monotonically decreasing loss over $3{,}000$ iterations -- though at the cost of more than $3.3$ hours of runtime. In contrast, parallel FISTA with a larger step size of $3\times 10^{-8}$ achieves a significantly better loss value with just $3$ iterations and $1$ minute. 

\medskip
With the step size $10^{-8}$, both methods yield further reductions in loss within reasonable iterations and times: parallel GPA with $1{,}079$ iterations and more than an hour, and parallel FISTA with $62$ iterations and less than $8$ minutes. Notably, the loss initially achieved by parallel GPA with step size $10^{-10}$ can now be matched by parallel GPA and FISTA after only $30$ and $10$ iterations, respectively. A visual comparison of the loss trajectories across these configurations is provided in Fig.~\ref{loss_gpa_fista_oa_connected_pic}.

\medskip
The experiments on WoS and OpenAlex instances demonstrate that parallel GPA and parallel FISTA are effective in reducing the objective loss, especially when step sizes are appropriately tuned. Interestingly, parallel FISTA can serve as a powerful heuristic acceleration method, particularly on large-scale datasets -- even without theoretical guarantees in the non-convex setting.

\medskip
\noindent{\bf Experiment 2: Do the solutions look reasonable?} 

Assessing whether a clustering solution ``looks reasonable'' is challenging for large graphs, as they cannot be easily visualized. 
%Unlike the 7-node graph from Example~1, larger graphs often produce dense, unintelligible plots. 
Example~\ref{ex_7node_graph} also demonstrated how GPA can yield a variety of outcomes depending on initialization, reflecting the non-convex nature of the objective and the sensitivity of the method to the starting point. To better assess the quality of the clustering solution, we focus here on the artificial graph, which has a comparable size to the mid-sized WoS graph but has two known ground-truth clusters (labeled 1 for Node IDs 1--500000 and labeled 2 for the others). 

With the same setting, step size $5\times 10^{-7}$ and $X^0$ with ones in the second row, both parallel GPA and parallel FISTA successfully recovered the expected structure in less than 10 iterations and about 15 seconds. This demonstrates the potential of the two algorithms when the underlying graph is well-formed.
\begin{figure}[h!]
  \centering
  \includegraphics[width=\linewidth]{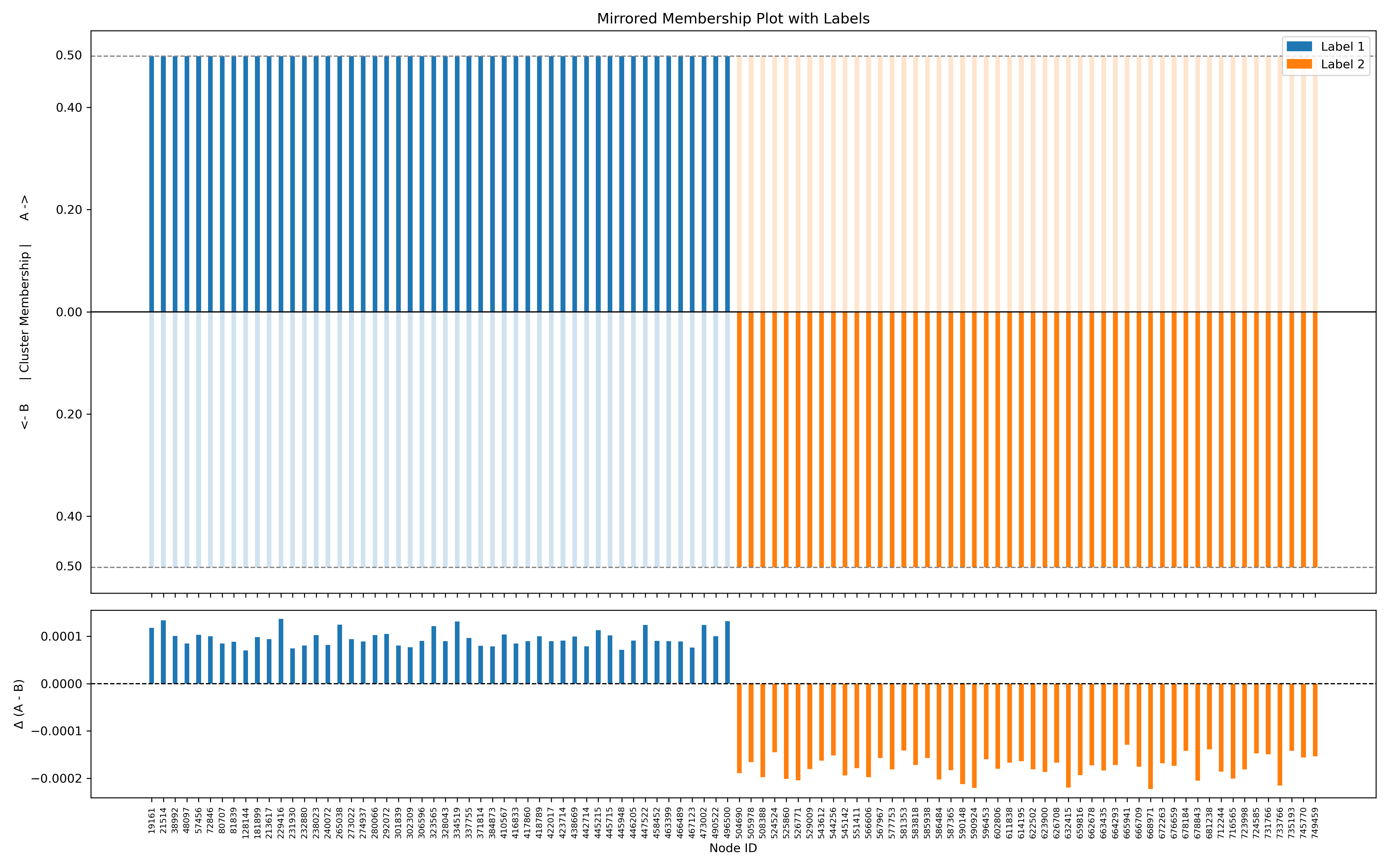}
  \caption{Clustering solution by parallel GPA on the artificial graph.}
  \label{membership_gpa_artifical700k_pic}
\end{figure}
Fig.~\ref{membership_gpa_artifical700k_pic} visualizes the clustering result (the output membership matrix $\bar X$) by parallel GPA for the artificial graph. Due to the graph’s size, we randomly selected 50 nodes from each of the two labeled groups for plotting. The top plot shows the membership values of each selected node with respect to Clusters A and B, where nondominant values are rendered in a fainter color. The bottom plot displays the differences $\Delta$ (A - B) in membership value between Cluster A and Cluster B. These plots indicate strong agreement between the dominant memberships and the ground-truth labels.

\medskip
While results on WoS and OpenAlex graphs are less conclusive -- likely due to the complex graph topology or sensitivity to initial points -- parallel GPA and FISTA remain promising for large-scale fuzzy clustering. %Further evaluation on real-world data is left for future work.

\section{Conclusion and future work}
We introduced parallel GPA and parallel FISTA, two algorithms for fuzzy clustering tailored to large-scale scientific article datasets. Both methods achieve substantial reductions in the clustering objective, with parallel FISTA notably delivering strong heuristic acceleration on real citation graphs containing hundreds of thousands to millions of nodes. For parallel GPA, we prove convergence to critical points and establish second-order optimality conditions that offer novel theoretical insights into solution quality. On a synthetic graph with a well-defined cluster structure, two methods successfully recover the expected clusters, confirming their effectiveness under ideal conditions.

A key innovation of our approach lies in the parallelization strategy: By taking advantage of the mathematical structure of the problem, projections are calculated independently on each column, while gradient updates are based on a small shared matrix \(XX^T\) among columns, avoiding explicit computation of the large \(X^TX\) matrix. This small matrix is also reused to evaluate the loss function \(f(X) = \|S - X^TX\|_F^2\), significantly reducing computational cost. These algorithmic improvements are implemented with CUDA on GPUs, exploiting their massive parallelism and high memory bandwidth to scale efficiently to very large datasets.

Future work will focus on improving initialization strategies, adaptive step size tuning, and exploiting hard clustering solutions to enhance robustness and interpretability. We also plan to use weighted edges, incorporate structured sparsity, metadata, and text-based similarity derived from titles and abstracts using large language models, to better capture complex real-world graphs. Finally, we plan to extensively compare the results with existing clusterings.

\bigskip
\noindent{\large \bf Acknowledgements.} This work is co-funded by the European Union (European Regional Development Fund EFRE, Fund No.\! STIIV-001) and supported by the German Competence Network for Bibliometrics (Grant~No. 16WIK2101A).
The research for this article was conducted at the Research Campus MODAL, funded by the German Federal Ministry of Education and Research (BMBF) (Grant No. 05M14ZAM, 05M20ZBM, 05M2025).

\end{document}